\documentclass[draft]{amsart}
\usepackage{url}
\usepackage{amssymb}
\usepackage{enumerate}
\usepackage{braket}

\theoremstyle{definition}
\newtheorem{definition}{Definition}[section]

\newtheorem{rem}[definition]{Remark}
\newtheorem*{ackname}{Acknowledgment}

\theoremstyle{theorem}
\newtheorem{thm}[definition]{Theorem}

\newtheorem{lem}[definition]{Lemma}
\newtheorem{prop}[definition]{Proposition}

\renewcommand{\phi}{\varphi}

\begin{document}

\title[Real zeros of Hurwitz zeta-functions in the interval $(0,1)$]
{Real zeros of  Hurwitz zeta-functions\\
and their asymptotic behavior\\
in the interval $(0,1)$}

\author[K.~Endo]{Kenta Endo}
\author[Y.~Suzuki]{Yuta Suzuki}



\maketitle
\begin{abstract}
Let $0<a\leq1, s\in\mathbb{C}$, and $\zeta(s,a)$ be the Hurwitz zeta-function. Recently, T.~Nakamura showed that $\zeta(\sigma,a)$ does not vanish for any $0<\sigma<1$ if and only if $1/2\leq a \leq1$. In this paper, we show that $\zeta(\sigma,a)$ has precisely one zero in the interval $(0,1)$ if $0<a<1/2$. Moreover, we reveal the asymptotic behavior of this unique zero with respect to $a$.
\end{abstract}

\section{Introduction and statement of main results}
Let $0<a \leq 1$ and $s \in \mathbb{C}$. The Hurwitz zeta-function $\zeta(s,a)$ is defined by
\[
\zeta(s,a)=\sum_{n=0}^\infty(n+a)^{-s},\quad s=\sigma+it,\quad \sigma>1,\quad t\in\mathbb{R}.
\]
This series converges absolutely and uniformly on any compact subset of the half-plane $\sigma>1$.
Moreover, $\zeta(s,a)$ has the meromorphic continuation to
the whole complex plane with a simple pole of residue $1$ at $s=1$.
Recently, T.~Nakamura~\cite{T. Nakamura} showed the following non-vanishing condition of $\zeta(\sigma,a)$ on the interval $(0,1)$.

\begin{thm}[see \cite{T. Nakamura}]\mbox{} 
\label{Thm:Nakamura}

\begin{enumerate}
\renewcommand{\labelenumi}{\upshape(\arabic{enumi})} \item The Hurwitz zeta-function $\zeta(\sigma,a)$ does not vanish in the interval $(0,1)$ when $1/2\leq a \leq 1$.
\item The Hurwitz zeta-function $\zeta(\sigma,a)$ has at least one zero in the interval $(0,1)$ when $0<a<1/2$.
\end{enumerate}
\end{thm}

In this paper, we show the following refinement of Theorem \ref{Thm:Nakamura}.

\begin{thm}
\label{main}
We have the following:
\begin{enumerate}
\renewcommand{\labelenumi}{\upshape(\arabic{enumi})}
\item The Hurwitz zeta-function $\zeta(\sigma,a)$ has precisely one zero in the interval $(0,1)$ when $0<a<1/2$, and these zeros are all simple.
\item For $0<a<1/2$, let $\beta(a)$ denote the unique zero of $\zeta(\sigma,a)$ in $(0,1)$.
Then $\beta:(0,1/2)\rightarrow(0,1)$ is a strictly decreasing $C^\infty$-diffeomorphism.
Moreover, the asymptotic formula
\[
\beta(a)=1-a+a^2\log a+O(a^2)\quad \text{as $a\rightarrow 0+$}
\]
holds.
\end{enumerate}
\end{thm}

The real zeros of various type of zeta-functions are studied by many researchers
since the information of such real zeros has significant meanings for arithmetical problems.
For example, the possible real zeros near to $s=1$ of Dirichlet $L$-functions associated to real characters
are well known as the Siegel zero, which are related to the distribution of prime numbers in arithmetic progressions
or the class number of the associated quadratic field.
Recall the well known formula
\[
L(s,\chi)=q^{-s}\sum_{r=1}^{q}\chi(r)\zeta(s,r/q),
\]
which relates Hurwitz zeta-functions and Dirichlet $L$-functions associated to characters modulo $q$.
In this context, our results can be seen as a toy model of the real zero problems on zeta-functions.

\begin{rem}
For the real negative zeros of $\zeta(s,a)$,
Nakamura~\cite{Nakamura2} obtained the condition of the existence of zeros in $(-1,0)$
and Matsusaka~\cite{Matsusaka} generalized Nakamura's condition to all intervals $(-N-1,-N)$ for integers $N\ge0$.
Their conditions are described with the roots of Bernoulli polynomials.
Moreover, Matsusaka~\cite{Matsusaka} proved that there is precisely one zero in the interval $[-2M-2,-2M)$ for $M\ge2$.
\end{rem}

\section{Proofs}
In order to prove Theorem $\ref{main}$, we define $H(a,x)$ by
\[
H(a,x):=\frac{e^{(1-a)x}}{e^x-1}-\frac{1}{x}=\frac{xe^{(1-a)x}-e^x+1}{x(e^x-1)},\quad x>0,\quad 0<a\leq1.
\]

\begin{lem}[see {\cite[Lemma 2.1]{T. Nakamura}}]\label{rep}
For $0<a\le 1$ and $0<\sigma<1$, we have 
\[
\Gamma(s)\zeta(s,a)=\int_{0}^{\infty}\left(\frac{e^{(1-a)x}}{e^x-1}-\frac{1}{x}\right)x^{s-1}dx=\int_{0}^{\infty}H(a,x)x^{s-1}dx.
\]
\end{lem}

\begin{lem}\label{Endo_points}
For any $0<a<\frac{1}{2}$, there exists a positive $x_{0}$ such that
\[
H(a,x_{0})=0,\ \ 
H(a,x)>0\ (\text{for $0< x< x_0$}),\ and \ 
H(a,x)<0\ \ (\text{for $x_0<x$}).
\]
\end{lem}
\begin{proof}
Since $x(e^x-1)>0$ for all $x>0$, it is enough to consider
\[
h(a,x):=x(e^x-1)H(a,x)=xe^{(1-a)x}-e^x+1.
\]
By differentiating $h(a,x)$ with respect to $x$, we have
\[
h'(a,x)=(1-a)xe^{(1-a)x}+e^{(1-a)x}-e^x.
\]
Put
\[
f(a,x):=e^{(a-1)x}h'(a,x)=(1-a)x+1-e^{ax}.
\]
By differentiating $f(a,x)$ with respect to $x$, we have
\[
f'(a,x)=1-a-ae^{ax},\quad
f''(a,x)=-a^2e^{ax}.
\]
Then we obtain that
\[
h(a,0)=0,\quad \lim_{x \to \infty}h(a,x)=-\infty,
\]
\[
h'(a,0)=0,\quad \lim_{x\to \infty}h'(a,x)=-\infty,
\]
\[
f(a,0)=0,\quad \lim_{x\to \infty}f(a,x)=-\infty,
\]
\[
f'(a,0)=1-2a>0, \quad \lim_{x\to \infty}f'(a,x)=-\infty,
\]
\[
f''(a,x)=-a^2e^{ax}<0.
\]
Hence we have the conclusion.
\end{proof}

\begin{lem}\label{Suzuki}
Let  $0<a<1/2$ and $x_0>0$ as in Lemma $\ref{Endo_points}$. Then the function
\[
x_0^{-\sigma}\Gamma(\sigma)\zeta(\sigma,a)
\]
is strictly decreasing for $0<\sigma<1$.
\end{lem}
\begin{proof}
The function in question can be rewritten as
\[
x_0^{-\sigma}\Gamma(\sigma)\zeta(\sigma,a)
=
\int_{0}^{x_0}H(a,x)\left(\frac{x}{x_0}\right)^{\sigma}\frac{dx}{x}
+
\int_{x_0}^{\infty}H(a,x)\left(\frac{x}{x_0}\right)^{\sigma}\frac{dx}{x}.
\]
As for the first integral, we have
\[
H(a,x)>0,\quad
\text{$\left(\frac{x}{x_0}\right)^{\sigma}$ is decreasing in $\sigma$},
\]
so that the integral
\[
\int_{0}^{x_0}H(a,x)\left(\frac{x}{x_0}\right)^{\sigma}\frac{dx}{x}
\]
is decreasing in $\sigma$.
The same kind of argument can be applied to the second integral, and we find that it is also decreasing for $\sigma$.
Thus we obtain the conclusion.
\end{proof}

\begin{proof}[Proof of {\upshape(1)} of Theorem \ref{main}]
Let $0<a<1/2$. Then we have
\[
\zeta(0,a)=\frac{1}{2}-a>0,\quad \lim_{\sigma\to 1-}\zeta(\sigma,a)=-\infty
\]
(see~\cite{Apostol,T. Nakamura}) which gives 
\[
\lim_{\sigma\to 0+}x_0^{-\sigma}\Gamma(\sigma)\zeta(\sigma,a)=\infty,\quad 
\lim_{\sigma\to 1-}x_0^{-\sigma}\Gamma(\sigma)\zeta(\sigma,a)=-\infty.
\]
By Lemma $\ref{Suzuki}$,
we have the uniqueness of zero since $x_0^{-\sigma}\Gamma(\sigma)$ has no zero and no pole in $(0,1)$.
Put
\[
F(\sigma,a)=x_0^{-\sigma}\Gamma(\sigma)\zeta(\sigma,a).
\]
To prove the simplicity of $\beta(a)$, we calculate $\frac{\partial}{\partial\sigma}F(\sigma,a)$.
For any $0<\sigma<1$, we have
\[
\frac{\partial}{\partial\sigma}F(\sigma,a)
=\left(\int_{0}^{x_0}+\int_{x_0}^{\infty}\right)H(a,x)\left(\frac{x}{x_0}\right)^\sigma\log\left(\frac{x}{x_0}\right)\frac{dx}{x}<0
\]
since, for any $x\neq x_0$, Lemma \ref{Endo_points} implies
\[
H(a,x)\left(\frac{x}{x_0}\right)^\sigma\log\left(\frac{x}{x_0}\right)<0.
\]
Since $x_0^{-\sigma}\Gamma(\sigma)$ has no zero and no pole in $(0,1)$, the simplicity of real zeros follows.
\end{proof}
Next we will prove (2) of Theorem \ref{main}.
\begin{lem}\label{conti}
Fix $0<\sigma<1$. Then the function in the variable $a$ defined by
\[
\int_{0}^{\infty}H(a,x)x^{\sigma-1}dx
\]
is continuous for $0<a<1$. Moreover, we have 
\[
\lim_{a\to0+}\int_{0}^{\infty}H(a,x)x^{\sigma-1}dx=\infty.
\]
\end{lem}
\begin{proof}
This result immediately follows from Lebesgue's convergence theorem.
\end{proof}

\begin{proof}[Proof of {\upshape(2)} of Theorem \ref{main}]
Let $0<a<1/2$ and $0<\sigma<1$. First, we show that $\beta$ is strictly decreasing. Suppose that $0<a_1<a_2<1/2$. By the integral representation of $\zeta(\sigma,a)$ in Lemma $\ref{rep}$, $\zeta(\sigma,a)$ is strictly decreasing with respect to $a$. Then it holds that
\[
0=\zeta\left(\beta(a_1),a_1\right)>\zeta\left(\beta(a_1),a_2\right).
\]
On the other hand, by the uniqueness of the zero of $\zeta(\sigma,a)$, it holds that
\[
0>\zeta(\sigma,a_2)\quad \text{if and only if}\quad \beta(a_2)<\sigma.
\]
Hence we have $\beta(a_1)>\beta(a_2)$. Therefore we obtain the monotonicity of $\beta$ and this implies that $\beta$ is injective.\par

Secondly, we show that $\beta$ is surjective. Take $\sigma\in(0,1)$. By Lemma \ref{conti}, we have
\[
\lim_{a\to 0+}\int_{0}^{\infty}H(a,x)x^{\sigma-1}dx=\infty
\]
and 
\[
\lim_{a\to 1/2}\int_{0}^{\infty}H(a,x)x^{\sigma-1}dx=\int_{0}^{\infty}H(1/2,x)x^{\sigma-1}dx<0
\]
since $H(1/2,x)<0$ 
for $x>0$ (see \cite[Lemma 2.2]{T. Nakamura}). Hence, there exists $a\in(0,1/2)$ such that $\int_{0}^{\infty}H(a,x)x^{\sigma-1}dx=0$. Therefore, $\beta$ is surjective.\par

Thirdly, in order to prove that $\beta$ is a $C^\infty$-diffeomorphism,
we regard $\zeta(\sigma,a)$ as a function of two variables $\zeta(\cdot,\cdot):(0,1)\times(0,1/2)\rightarrow\mathbb{R}$. 
Then, $\zeta(\cdot,\cdot)$ is a $C^\infty$-function and, by Theorem \ref{main} (1),
we have $\frac{\partial}{\partial \sigma}\zeta\left(\beta(a),a\right)\neq0$,
Hence by the implicit function theorem and the uniqueness of zeros,
we have that $\beta:(0,1/2)\rightarrow(0,1)$ is $C^\infty$.
By Lemma 2.1, we find that
\[
\frac{\partial}{\partial a}\zeta(\sigma,a)=-\frac{1}{\Gamma(s)}\int_0^\infty\frac{e^{(1-a)x}}{e^x-1}x^{\sigma}dx<0.
\]
Thus by the inverse function theorem, we find that the inverse of $\beta$ is also $C^\infty$.

Finally, we shall show the asymptotic formula for $\beta(a)$ as $a\rightarrow 0+$. Note that it follows that $\beta(a)\rightarrow1-$ as $a\rightarrow0+$ from the surjectivity and the monotonicity of $\beta(a)$. Thus we can assume, without loss of generality, that $\beta(a)\ge1/2$.
We start with the Euler-Maclaurin summation formula
\begin{align*}
&\zeta(s,a)=a^{-s}+\sum_{n=0}^{\infty}\frac{1}{(n+a+1)^s}\\
=\,&a^{-s}+\frac{(a+1)^{1-s}}{s-1}+\frac{(a+1)^{-s}}{2}
-s\int_{0}^{\infty}\left(x-[x]-\frac{1}{2}\right)(x+a+1)^{-s-1}dx
\end{align*}
which holds for $0<\sigma<1$. When $s=\sigma$, it holds that
\[
\left|\sigma\int_{0}^{\infty}(x-[x]-\frac{1}{2})(x+a+1)^{-\sigma-1}dx\right|
\ll\sigma\int_0^{\infty}(x+a+1)^{-\sigma-1}dx\ll1
\]
uniformly for $0<\sigma<1$. Thus we obtain
\[
\zeta(\sigma,a)=a^{-\sigma}+\frac{(a+1)^{1-\sigma}}{\sigma-1}+O(1).
\]
Let us use an abbreviation $\beta=\beta(a)$ and substitute $\sigma=\beta$. Then we have
\begin{equation}
\label{EQ:basic_approx}
\beta-1=-a^{\beta}(a+1)^{1-\beta}+O((1-\beta)a^{\beta}).
\end{equation}
Note that $\beta-1\ll a^\beta\ll a^{1/2}$ by the assumption $\beta\ge1/2$.
Hence it holds that
\begin{align}
a^\beta
&=a\exp((\beta-1)\log a)
=a+(\beta-1)a\log a +O\left((\beta-1)^2a|\log a|^2\right).\label{EQ:approx1} 
\end{align}
In particular we have 
\begin{equation}\label{EQ:approx2}
a^\beta\ll a.
\end{equation}
It also holds that
\begin{align}
(a+1)^{1-\beta}&=\exp\left((1-\beta)\log(1+a)\right)
=1+O\left((1-\beta)a\right).\label{EQ:approx3}
\end{align}
By substituting the estimates \eqref{EQ:approx1}, \eqref{EQ:approx2} and \eqref{EQ:approx3} into \eqref{EQ:basic_approx}, we have
\begin{align}
\beta-1
&=-a+(1-\beta)a\log a+O\left((1-\beta)a\right).
\label{EQ:second_approx}
\end{align}
In particular, we have
\[
1-\beta\ll a\quad \text{and} \quad \beta-1=-a+O\left(a^2|\log a|\right).
\]
Substituting this estimate into \eqref{EQ:second_approx}, we have
\begin{align*}
\beta-1&=-a+\left(a+O(a^2|\log a|)\right)a\log a+O(a^2)\\
&=-a+a^2\log a+O(a^2).
\end{align*}
This completes the proof.
\end{proof}

\begin{rem}
We can also obtain the asymptotic formula for $\beta(a)$ in Theorem $\ref{main}$ by using generalized Stieltjes constants (see e.g.~\cite{Berndt}). Let
\[
\zeta(s,a)=\frac{1}{s-1}+\sum_{n=0}^{\infty}\gamma_n(a)(s-1)^n
\]
be the Laurent series of $\zeta(s,a)$ at $s=1$.
Then $\beta$ can be written by
\begin{equation}
\label{EQ:Berndt_expansion}
\beta(a)=1-\frac{1}{\gamma_0(a)}\left(1+\sum_{n=1}^{\infty}\gamma_n(a)(\beta-1)^{n+1}\right).
\end{equation}
Hence, in principle, we can obtain further asymptotic expansions of $\beta(a)$
from asymptotic expansions of Stieltjes constants with respect to the variable $a$
by using the formula \eqref{EQ:Berndt_expansion} recursively.
\end{rem}

\begin{ackname}
The authors are deeply grateful to Prof.~Kohji Matsumoto, Prof.~Takashi Nakamura, Mr.~Tomohiro Ikkai, Mr.~Sh\={o}ta Inoue, Mr.~Masahiro Mine for their helpful comments.
\end{ackname}

\vspace{1mm}
\begin{flushleft}
{\footnotesize
{\sc
Graduate School of Mathematics, Nagoya University,\\
Chikusa-ku, Nagoya 464-8602, Japan.
}\\
{\it E-mail address}, K. Endo: {\tt m16010c@math.nagoya-u.ac.jp}\\
{\it E-mail address}, Y. Suzuki: {\tt m14021y@math.nagoya-u.ac.jp}\\
}
\end{flushleft}

\begin{thebibliography}{9}
\bibitem{Apostol}
T. M. Apostol, \emph{Introduction to Analytic Number Theory}, Undergraduate Texts in Mathematics, Springer, New York, 1976.
\bibitem{Berndt}B. C. Berndt, \emph{On the Hurwitz zeta-function}, Rocky Mountain J. Math. \textbf{2} (1972), no. 1, 151-157.
\bibitem{Matsusaka}
T. Matsusaka, \emph{Real zeros of the Hurwitz zeta function}, \texttt{arXiv/1610.07945}, preprint (2016).
\bibitem{T. Nakamura}
T. Nakamura,
\emph{Real zeros of Hurwitz-Lerch zeta and Hurwitz-Lerch type of Euler-Zagier double zeta functions},
Math. Proc. Cambridge Philos. Soc. \textbf{160} (1) (2016), 39--50.
\bibitem{Nakamura2}
T. Nakamura,
\emph{Real zeros of Hurwitz-Lerch zeta functions in the interval $(-1,0)$},
J. Math. Anal. Appl. \textbf{438} (1) (2016), 42--52.
\bibitem{Spira}
R. Spira, \emph{Zeros of Hurwitz zeta functions}, Math. Comp. \textbf{30} (136) (1976), 863-866.
\end{thebibliography}
\end{document}